\documentclass[12pt]{amsart}
\usepackage[dvips]{graphicx}
\usepackage{amsmath,graphics}
\usepackage{amsfonts,amssymb}
\usepackage{xypic}
\usepackage{comment}
\usepackage{color}
\specialcomment{proofc}{}{}
\includecomment{proofc}
\theoremstyle{plain}
\newtheorem*{theorem*}{Theorem}
\newtheorem*{lemma*} {Lemma}
\newtheorem*{corollary*} {Corollary}
\newtheorem*{proposition*}{Proposition}
\newtheorem*{conjecture*}{Conjecture}
\newtheorem{theorem}{Theorem}[section]
\newtheorem{lemma}[theorem]{Lemma}
\newtheorem*{theorem1*}{Theorem 1}
\newtheorem*{theorem2*}{Theorem 2}
\newtheorem*{theorem3*}{Theorem 3}

\newtheorem{question}[theorem]{Question}

\newcommand\rg{\mathrm{rg}\,}
\newcommand\hg{\mathrm{hg}\,}
\newcommand{\co}{\colon \thinspace}

\theoremstyle{remark}
\newtheorem*{remark}{Remark}
\newtheorem*{definition}{Definition}

\newtheorem{example*}{Example}
\newtheorem*{claim}{Claim}

\theoremstyle{definition}

\def\op{\operatorname}

\def\G{\Gamma}

\def\gl{\mbox{GL}} \def\Q{\Bbb{Q}}  \def\Z{\Bbb{Z}}  
\def\N{\Bbb{N}}   \def\ll{\langle} \def\rr{\rangle}
 \def\a{\alpha}  \def\tor{\mbox{Tor}} \def\bp{\begin{pmatrix}}
\def\sm{\setminus} \def\ep{\end{pmatrix}} \def\bn{\begin{enumerate}} 
   \def\en{\end{enumerate}}
\def\ba{\begin{array}} \def\ea{\end{array}}  
 \def\S{\Sigma}  \def\a{\alpha}  \def\ti{\tilde}

\def\ker{\mbox{Ker}}\def\be{\begin{equation}} \def\ee{\end{equation}} 
   
 \def\hom{\mbox{Hom}}  
   
 \def\dim{\mbox{dim}}

 \def\qt{\Q[t^{\pm 1}]}   \def\rk{\op{rk}}

\def\tpm {[t^{\pm 1}]}

\def\ol{\overline}

\def\wti{\widetilde}

\begin{document}
\title{Rank gradients of infinite cyclic covers of 3-manifolds}
\author{Jason DeBlois}
\address{Department of Mathematics, University of Pittsburgh}
\email{jdeblois@pitt.edu}
\thanks{J.~DeBlois was partially supported by NSF grant DMS-1240329}

\author{Stefan Friedl}
\address{Fakult\"at f\"ur Mathematik\\ Universit\"at Regensburg\\   Germany}
\email{sfriedl@gmail.com}

\author{Stefano Vidussi}
\address{Department of Mathematics, University of California,
Riverside, CA 92521, USA} \email{svidussi@math.ucr.edu} \thanks{S. Vidussi was partially supported by NSF grant
DMS-0906281.}
\date{\today}

\begin{abstract}
Given a 3-manifold $M$ with no spherical boundary components, and a primitive class $\phi\in H^1(M;\Z)$, we show that the following are equivalent:
\bn
\item $\phi$ is a fibered class,
\item the  rank gradient of $(M,\phi)$ is zero,
\item  the Heegaard gradient of $(M,\phi)$ is zero.
\en
\end{abstract}

\maketitle

\section{Introduction}

A  \emph{directed $3$-manifold} 
is a pair $(M,\phi)$ where  $M$ is a compact, orientable, connected 3--manifold with toroidal or empty boundary,
and $\phi\in H^1(M;\Z)=\hom(\pi_1(M),\Z)$ is a primitive class, i.e. $\phi$ viewed as a homomorphism $\pi_1(M)\to \Z$ is an epimorphism.
We say that a directed 3-manifold
\emph{$(M,\phi)$ fibers over $S^{1}$}
 if there exists
 a fibration $p\co M\to S^1$ such that the induced map $p_*\co \pi_1(M)\to \pi_1(S^1)=\Z$ coincides with $\phi$.
 We refer to such $\phi$ as a \emph{fibered class}.

It is well-known that the pair $(\pi_1(M),\phi\co \pi_1(M)\to \Z)$ determines whether $\phi$ is fibered or not.
 Indeed, it follows from Stallings' theorem \cite{St62} (together with the resolution of the Poincar\'e conjecture) that $\phi$ is a fibered class if and only if $\ker(\phi\co \pi_1(M)\to \Z)$ is finitely generated.

Stallings' theorem can be generalized in various directions (see e.g.
 \cite[Theorem~5.2]{FV12}, \cite{SW09a,SW09b} and \cite{FSW13}).
Our main result gives a new  fibering criterion which is also a strengthening
 of Stallings' theorem.
In order to state our result we need the notion of rank gradient which was first introduced by M. Lackenby \cite{La05}.
Given a finitely generated group $\pi$  we denote by $\rk(\pi)$ the \emph{rank} of $\pi$, i.e. the minimal number of generators of $\pi$.
If $(M,\phi)$ is a directed $3$-manifold then we write
\[\pi_n=\ker(\pi_1(M)\xrightarrow{\phi}\Z\to \Z/n),\]
and we refer to
\[ \rg(M,\phi):=\liminf_{n\to \infty} \frac{1}{n}\rk(\pi_n)\]
as the \emph{rank gradient} of $(M,\phi)$.  (In the notation of \cite{La05} this is the rank gradient of $(\pi_1 M,\{\pi_n\})$.)

If $\phi$ is a fibered class, then  $\phi$ is dual to a fiber $S$ of a fibration $M\to S^1$ (a connected surface) and 
it is straightforward to show
that  $\rk(\pi_n) \leq 1+\op{genus}(S)$ for any $n$ (see e.g. Lemma \ref{lem:rgbounded}). In particular $\rg(M,\phi)=0$.

Our main result now says that the converse to this statement holds.
More precisely, we will prove the following theorem.

\begin{theorem}\label{mainthm}
Let $(M,\phi)$ be a directed $3$-manifold. Then the following three statements are equivalent:
\bn
\item $\phi$ is fibered,
\item the sequence $\rk(\pi_n)$,  $n\in \N$ is bounded,
\item $\rg(M,\phi)=0$.
\en
\end{theorem}

It follows from the discussion preceding the theorem that it suffices to prove that (3) implies (1).
In fact we will present two quite different approaches to the proof of this statement.

The first, discussed in Section \ref{section:proof1},  uses tools from geometric group theory: acylindrical accessibility and the finite height property.  It applies only to closed hyperbolic manifolds but has the advantage of generalizing more readily to the broader setting of hyperbolic groups, where the separability results used for the general case are not currently available. Moreover, with more work, this approach yields explicit lower bounds on the rank gradient.  In the sequel \cite{De13} the first author refines Theorem \ref{mainthm} in this way for $M$ closed and hyperbolic, bounding $\rg(M,\phi)$ below in terms of the Thurston norm of a non-fibered class $\phi$.

The second proof, discussed in Section \ref{section:proof2}, uses  the  recent proof (see \cite{FV12}) that
given any non-fibered directed 3-manifold $(M,\phi)$ there exists a twisted Alexander polynomial which vanishes.
This proof in turn relies on the recent results of  D. Wise \cite{Wi09,Wi12a,Wi12b}.

To describe our second result, we need to introduce the notion of Heegard gradient. A \emph{Heegaard surface} for a compact $3$-manifold $M$ is an embedded separating surface $S\subset M$
such that the two components of $M$ cut along $S$ are compression bodies.
The minimal genus of a Heegaard surface is called the \emph{Heegaard genus} $h(M)$ of $M$.
Given a class $\phi \in H^1(M;\Z)=\hom(\pi_1(M),\Z)$ we can then define the \emph{Heegaard gradient $hg(M,\phi)$} in a similar fashion to the rank gradient.
We refer to Section \ref{section:hg} for more details. In that section we will also see that
the subsequent theorem is  a straightforward consequence of Theorem \ref{mainthm}.

\begin{theorem}\label{mainthmheegaard}
Let $(M,\phi)$ be a  directed $3$-manifold. Then the following three statements are equivalent:
\bn
\item $\phi$ is fibered,
\item the sequence $h(M_n)$,  $n\in \N$ is bounded,
\item $\hg(M,\phi)=0$.
\en
\end{theorem}

This theorem was proved by M. Lackenby \cite[Theorem~1.11]{La06}
for closed hyperbolic $3$-manifolds. To the best of our knowledge the general case has not been proved before.  The equivalence presented in the abstract immediately follows from  Theorems 1.1 and 1.2 if $M$ has empty or toroidal boundary.  In the general case see Lemma \ref{lem:extendmainthm}.

The equivalence between vanishing of rank and Heegaard gradients holds with no restriction on boundaries.  This is proved at the end of Section \ref{section:proof2}.

\begin{theorem}\label{rankvsheegaard}
For a compact, orientable, connected $3$-manifold $M$ and a primitive class $\phi\in H^1(M;\mathbb{Z})$, $\rg(M,\phi)=0$ if and only if $\hg(M,\phi)=0$.\end{theorem}

We will now formulate the last theorem of the paper. Recall that a group $\pi$ is normally generated by a subset $S\subset \pi$ if $\pi$ is the smallest normal subgroup of $\pi$ which contains $S$. We define the \emph{normal rank $n(\pi)$ of $\pi$} to be the smallest cardinality of a normal generating set of $\pi$. The first part of this theorem can also be viewed as a strengthening of Stallings' fibering theorem.

\begin{theorem}\label{mainthmnormal}
\bn
\item If  $(M,\phi)$  is a non-fibered directed $3$-manifold, then $\ker(\phi)$ admits a finite index subgroup with infinite normal rank.
\item There exists a non-fibered directed $3$-manifold $(M,\phi)$, such that $\ker(\phi)$ has finite normal rank.
\en
\end{theorem}

\subsection*{Convention.} Unless it says specifically otherwise, all groups are assumed to be finitely generated, all manifolds are assumed to be orientable, connected and compact, and all 3-manifolds are assumed to have empty or toroidal boundary.

\subsection*{Acknowledgment.} We wish to thank Dan Silver for a helpful conversation. We are also grateful to the referee for many helpful comments.

\section{The rank gradient and the Heegaard gradient}

\subsection{The rank gradient}

We start out with the following well-known lemma:

\begin{lemma}
Let $G$ be a finitely generated group and let $H\subset G$ be a finite index subgroup,
then  
\be \label{equ:rs} \rk(H)\leq [G:H]\cdot (\rk(G)-1)+1\leq [G:H]\cdot \rk(G).\ee
\end{lemma}

\begin{proof}
Let $\alpha\colon F\to G$ be an epimorphism where $F$ is a free group of rank $\rk(G)$. Note that $\alpha^{-1}(H)$ is a subgroup of $F$ of index $d:=[G:H]$. It  follows from elementary properties of the free group that $\alpha^{-1}(H)$ is a free group 
of rank 
\[ d\cdot (\rk(F)-1)=[G:H]\cdot (\rk(G-1)).\]
Since $\alpha$ restricts to an epimorphism from the free group $\alpha^{-1}(H)$ onto $H$ it now follows that 
\[ \rk(H)\leq [G:H]\cdot (\rk(G)-1)+1\leq [G:H]\cdot \rk(G).\]
\end{proof}

We now let $\pi$ be a finitely generated group and let  $\phi\co \pi\to \Z$ be a homomorphism then we write
\[\pi_n:=\ker(\pi_1(M)\xrightarrow{\phi}\Z\to \Z/n),\]
and we refer to
\[ \rg(\pi,\phi):=\liminf_{n\to \infty} \frac{1}{n}\rk(\pi_n)\]
as the \emph{rank gradient} of $(\pi,\phi)$.
It is a consequence of (\ref{equ:rs}) that this limit does indeed exist.
(Note that Lackenby defines the rank gradient using $\frac{1}{n}(\rk(\pi_n)-1)$ instead of $\frac{1}{n}\rk(\pi_n)$, but it is clear that this gives rise to the same limit.)

The following lemma is now an immediate consequence of (\ref{equ:rs}) and the definitions:

\begin{lemma}\label{lem:rg}
Let $\pi$ be a finitely generated group and let  $\phi\co \pi\to \Z$ be a homomorphism.
\bn
\item If $\a\co \G\to \pi$ is an epimorphism, then
\[ \rg(\G,\phi\circ \a)\geq \rg(\pi,\phi).\]
\item If $\G\subset \pi$ is a finite index subgroup, then
\[ \rg(\G,\phi)\leq [\pi:\G]\cdot \rg(\pi,\phi).\]
\en
\end{lemma}

The following two lemmas show that Theorem \ref{mainthm} is indeed a strengthening
 of Stallings' fibering theorem.

\begin{lemma}\label{lem:rgbounded}
Let $\pi$ be a finitely generated group and let  $\phi\co \pi\to \Z$ be an epimorphism.
If $\ker(\phi)$ is  generated by $k$ elements, then  for any $n\in \N$ we have $\rk(\pi_n)\leq k+1$, in particular $\rg(\pi,\phi)=0$.
\end{lemma}

\begin{proof}
We write $K=\ker(\phi)$.
Note that the epimorphism $\phi\colon \pi\to \Z=\ll t\rr$ splits since $\ll t\rr$ is in particular a free group.
We can thus  view $\pi$ as a semidirect product  $\pi=\ll t\rr \ltimes K$.
Under this identification we furthermore have  that $\pi_n= \ll t^n\rr \ltimes K$. In particular if $\{g_1,\dots,g_k\}$ is a generating set for $K$,
then $\{t^n,g_1,\dots,g_k\}$ is a generating set for $\pi_n$.
\end{proof}

\begin{lemma}
There exists a finitely presented group $\pi$ and an epimorphism $\phi\colon \pi\to \Z$ such that $\ker(\phi)$
is infinitely generated, but such that $\rk(\pi_n)\leq 2$ for all $n$.
\end{lemma}

\begin{proof}
We consider the semidirect product
\[ \pi:=\ll t\rr \ltimes \Z[1/2]\]
where $t^n$ acts on $\Z[1/2]$ by multiplication by $2^n$ together with the epimorphism $\phi\colon \pi \to \Z$
which is defined by $\phi(t^n)=n$ and $\phi(a)=0$ for $a\in \Z[1/2]$.
It is clear that $\ker(\phi)=\Z[1/2]$ is not finitely generated.
On the other hand it is straightforward to see that
\[ \pi_n=\ll t^n\rr \ltimes \Z[1/2]\]
is generated by $t^n$ and $1\in \Z[1/2]$. We thus showed that $\rk(\pi_n)\leq 2$ for all $n$.
\end{proof}

This raises the following question.

\begin{question}
Does there exist  a finitely presented group $\pi$  and  a homomorphism   $\phi\co \pi\to \Z$ such that $\rg(\pi,\phi)=0$
but such that the sequence  $\rk(\pi_n)$ is unbounded?
\end{question}

We conclude this section with the following elementary lemma:

\begin{lemma}\label{lem:rgfree}
Let $F$ be a free group on $k$ generators and $\phi\co F\to \Z$ an epimorphism, then
\[ \rg(F,\phi)=k-1.\]
\end{lemma}

The statement of the lemma already appears in \cite{La05}, but for the reader's convenience we provide a proof.

\begin{proof}
It is well-known that any subgroup of $F$ of index $n$ is a free group on $n(k-1)+1$ generators.
(This follows for example for an elementary argument using Euler characteristics of finite covers of graphs.)
The lemma is now an immediate consequence of this observation.
\end{proof}

\subsection{The Heegaard gradient}\label{section:hg}

We now recall several basic definitions and facts on Heegaard splittings of $3$-manifolds.
We refer to \cite{Jo} and \cite{Sc02} for more details.
We start out with several definitions:
\bn
\item
 A \emph{compression body $H$} is the result of gluing disjoint $2$-handles to $\S\times [0,1]$,
 where $\S$ is a closed surface,
along $\S\times 1$ and then capping off some spherical boundary components with $3$-balls.
We then write $\partial_+ H=\S\times 0$ and $\partial_-H=\partial H\sm \partial_+H$.
Note that a compression body with $\partial_-H=\emptyset$ is a handlebody.
\item A \emph{Heegaard surface} for a $3$-manifold $M$  is an embedded separating surface $S\subset M$
such the two components of $M$ cut along $S$ are compression bodies $H_1$ and $H_2$ with $\partial_+ H_1=\S=\partial_+ H_2$.
\en
Note that every compact $3$-manifold admits a Heegaard surface (see e.g. \cite[Section~2]{Sc02}).
In the following we refer to the minimal genus of a Heegaard surface as the \emph{Heegaard genus $h(M)$} of $M$.

Furthermore, given a directed $3$-manifold $(M,\phi)$ with corresponding cyclic covers $M_n, n\in \N$ we define,
following \cite{La06}, the \emph{Heegaard gradient} of $(M,\phi)$ to be
\[ \hg(M,\phi):=\liminf_{n\to \infty} \frac{1}{n}h(M_n).\]
Note that if $p\colon \wti{M}\to M$ is a $k$-fold cover, then the preimage of a Heegaard surface is again a Heegaard surface, it now follows easily that $h(\wti{M})\leq k\cdot h(M)$. We therefore see  in particular that the Heegaard gradient is well-defined.

We summarize a few key properties of the Heegaard genus in a lemma.

\begin{lemma}\label{lem:hg}
Let $M$ be a $3$-manifold, then the following hold:
\bn
\item
\[\rk(\pi_1(M))\leq \left\{ \ba{ll} h(M), &\mbox{ if $M$ is closed}, \\ 2h(M), &\mbox{ otherwise.}\ea \right.\]
\item If $\phi\in H^1(M;\Z)$ is a primitive class, then
\[ \rg(\pi_1(M),\phi)\leq \left\{ \ba{ll} hg(M), &\mbox{ if $M$ is closed}, \\ 2hg(M), &\mbox{ otherwise.}\ea \right.\]
\item If $\phi\in H^1(M;\Z)$ is a primitive fibered class, then
\[ h(M)\leq 2\cdot \mbox{genus of the fiber}+1.\]
\en
\end{lemma}

\begin{remark}
\bn
\item
Note that there exist closed $3$-manifolds with $\rk(\pi_1(M))<h(M)$.
In fact there exist examples of such $3$-manifolds which are Seifert fibered \cite{BZ84}, graph manifolds \cite{We03}, \cite{ScWe07}
and hyperbolic \cite{Li11}. On the other hand  J. Souto \cite{So08} and H. Namazi--J. Souto \cite{NS09} showed that $\rk(\pi_1(N))=h(N)$
 for hyperbolic $3$-manifolds that are  `sufficiently complicated' in a certain sense.
\item  To the best of our knowledge it is not known whether there exist closed $3$-manifold
pairs $(M,\phi)$ with $\rg(M,\phi)<\hg(M,\phi)$.
\item Note that Theorem \ref{mainthmheegaard} is  an immediate consequence of Theorem \ref{mainthm} and Lemma \ref{lem:hg}.
\en
\end{remark}

\begin{proof}
First note that if $M$ is a closed $3$-manifold and $\S$ is a Heegaard surface of genus $g$,
then the compression bodies obtained by cutting $M$ along $\S$ are in fact handlebodies.
We can thus view $M$ as the result of gluing together two handlebodies $H_1,H_2$ with $g$ 1-handles each.
In particular we can build $M$ out of $H_1$ by adding $g$ 2-handles and one 3-handle. Since $\pi_1(H_1)$ is generated by $g$ elements it follows that $\rk(\pi_1(M))\leq g$. This evidently implies (1) and (2) for closed $3$-manifolds.

If $M$ is any  $3$-manifold and $\S$ is a Heegaard surface of genus $g$,
then we can view $M$ as the result of gluing 2-handles and 3-handles to $\S\times [-1,1]$.
It follows that $\pi_1(M)$ is generated by a generating set for $\pi_1(\S)$, i.e. $\rk(\pi_1(M))\leq 2g$. This evidently implies (1) and (2) for
 $3$-manifolds which are not closed.

We now turn to the proof of (3).
Suppose that $\S$ is the fiber of a fibration $M\to S^1$.
We can then identify $M$ with $(\S\times [0,1])/(x,0)\sim (f(x),1))$ for some self-diffeomorphism $f$ of $\S$.
If $M$ is a closed $3$-manifold, then we pick two disjoint disks $D_1$ and $D_2$ on $\S$.  Then
\[ (\S\sm (D_1\cup D_2)) \times 0\,\,\cup\,\, (\S\sm (D_1\cup D_2)) \times \frac{1}{2}\,\,\cup \,\,\partial D_1\times [0,\frac{1}{2}]\,\,\cup\,\, \partial D_2\times [\frac{1}{2},1]\]
is a surface of genus $2g+1$ and it is in fact a Heegaard surface for $M$: it cuts $M$ into
\[ \ba{l}
(\Sigma - \mathit{int}\,D_1)\times[0,1/2]\cup (\mathit{int}\,D_2)\times [1/2,1],\mbox{ and }\\
(\Sigma-\mathit{int}\,{D}_2)\times[1/2,1]\cup (\mathit{int}\,D_1)\times [0,1/2],\ea \]
 each the union of a $1$-handle with a handlebody of the form (bounded surface)$\times$(interval).

If $M$ is not closed then $\Sigma$ has non-trivial boundary and $M$ has toroidal boundary $\partial \Sigma\times[0,1]/(x,0)\sim (f(x),1))$.  Let $\eta$ be a closed, $f$-invariant tubular neighborhood of $\partial \Sigma$ in $\Sigma$, so $N = (\eta\times[0,1])/(x,0)\sim (f(x),1))$ is a tubular neighborhood of $\partial M$ in $M$, and let $D_1$ be a disk in $\Sigma$ disjoint from $\eta$.  Taking $H_1 = (\Sigma - \mathit{int}\,{D}_1)\times [0,1/2] \cup N$, we claim that the frontier $S$ of $H_1$ in $M$ is a Heegaard surface.

A maximal collection of disjoint, non parallel, non boundary-parallel arcs embedded in $\Sigma - (\mathit{int}\,{D}_1\sqcup\eta)$ that each begin and end on $\partial D_1$, gives rise, by crossing with $[0,\frac{1}{2}]$, to a collection $\mathcal{D}$ of disjoint compressing disks for $S$ in $H_1$ with the property that $H_1 - (S\cup\bigcup \{D\in\mathcal{D}\})$ retracts onto $\partial M$.  Thus $H_1$ is a compression body.
It is easy to see that the other side $H_2$ of $S$ in $M$ is a handlebody of the form described in the closed case, and the claim follows.
\end{proof}


\section{Proof of the main theorem for closed hyperbolic $3$-manifolds} \label{trees}\label{section:proof1}

Given a finitely generated group $\Gamma$ acting
on a tree $T$, an ``accessibility'' principle relates the combinatorics of $\Gamma\backslash T$ to the structure of $\Gamma$.  \textit{Acylindrical} accessibility, introduced by Z. Sela \cite{Sela}, does not require prior knowledge of the structure of vertex or edge stabilizers, but only that their action on $T$ is ``nice enough'':

\begin{definition} The action $\Gamma \times T \to T$ is \textit{$k$-acylindrical} if no $g \in \Gamma - \{1\}$ fixes a segment of length greater than $k$, and \textit{$k$-cylindrical} otherwise.  \end{definition}

We will later on make use of the following  theorem of R.~Weidmann.

\begin{theorem}[Weidmann, \cite{Weidmann}] \label{thm:weidmann} Let $\Gamma$ be a non-cyclic, freely indecomposable, finitely generated group and $\Gamma \times T \to T$ a $k$-acylindrical, minimal (i.e.~leaving no proper subtree invariant) action.  Then $\Gamma \backslash T$ has at most $1+2k(\rk(\Gamma) - 1)$ vertices.  \end{theorem}

We will use the height of edge stabilizers, a notion from \cite{GMRS}, to bound cylindricity of the action under consideration.

\begin{definition} The \textit{height} of an infinite subgroup $\Lambda$ in $\Gamma$ is $k$ if there is a collection of $k$ essentially distinct conjugates of $\Lambda$ such that the intersection of all the elements of the collection is infinite and $k$ is maximal possible.  (The conjugate of $\Lambda$ by $\gamma$ is \emph{essentially distinct} from the conjugate by $\gamma'$ if $\Lambda\gamma\neq\Lambda\gamma'$.) \end{definition}

\begin{lemma}\label{cylinders vs height}  Suppose a torsion-free group $\Gamma$ acts on a tree $T$ on the left, transitively on edges.  If the stabilizer $\Lambda$ of an edge $e_0$ has height $k$ in $\Gamma$ then the action of $\Gamma$ on $T$ is $k$-acylindrical.\end{lemma}

\begin{proof}  Because the action is transitive on edges, each edge stabilizer is conjugate to $\Lambda$, and the conjugates corresponding to distinct edges are essentially distinct: for an edge $e\neq e_0$, the stabilizer of $e$ in $\Gamma$ is $\gamma^{-1}\Lambda\gamma$, where $\gamma\in \Gamma$ satisfies $\gamma.e = e_0$.  Every element $\lambda\gamma$ of $\Lambda\gamma$ thus also satisfies $(\lambda\gamma).e = e_0$.

Now suppose $\gamma\in\Gamma-\{1\}$ fixes an edge arc of length $n$.  Then $\gamma$, hence also the subgroup $\langle\gamma\rangle$ that it generates, is in the intersection of the conjugates corresponding to the edges of this arc.  Since $\Gamma$ is torsion-free, $\langle\gamma\rangle$ is infinite and $\Lambda$ has height at least $n$.  But $\Lambda$ has height $k$, so it follows that $\Gamma\times T\to T$ is $k$-acylindrical.\end{proof}

Let $(M,\phi)$ be a directed $3$-manifold. We pick a properly embedded oriented surface $S$ in $M$ dual to $\phi$
of minimal complexity. (Here, recall that the complexity of  a surface $S$ with connected components $S_1\cup\dots \cup S_k$ is defined as
$\chi_-(S)=\sum_{i=1}^k \max\{-\chi(S_i),0\}$.)
We also pick a tubular neighborhood $S\times [-1,1]$ of $S$ in $M$.

We view $S^1$ as the topological space underlying a graph $G$, with a single vertex $v$  and a single edge $e$.
Note that there exists a canonical continuous map $p\co M\to G$ given by sending $S\times (-1,1)\to (-1,1)\to e$
and by sending every point in $M\sm S\times (-1,1)$ to $v$.
The induced map $p_*\co \pi_1(M)\to \pi_1(G)=\Z$ is precisely the map given by $\phi\in H^1(M;\Z)=\hom(\pi_1(M),\Z)$.

We denote by $G_0$ the graph which has one vertex for each component of $M\sm S\times (-1,1)$ and one edge for each component of $S\times [-1,1]$
with the obvious attaching maps. Note that there exist canonical maps $q\co M\to G_0$
 and $r\co G_0\to G$
which make the following diagram commute:
\[ \xymatrix{&M\ar[dl]_q\ar[dr]^p\\ G_0\ar[rr]^r&& G.}\]
It is clear from the definitions that all the maps induce epimorphisms on fundamental groups.
In particular  $G_0$ is not a tree and hence its Euler characteristic $\chi(G_0)$ is non-positive.  If $\chi(G_0)$ is negative the conclusion of Theorem \ref{mainthm} requires no machinery.

\begin{lemma}\label{disconnected case}
Let $(M,\phi)$ be a directed $3$-manifold. If the graph $G_0$  has $\chi(G_0)<0$, then $\rg(M,\phi) \geq -\chi(G_0)$.
\end{lemma}

\begin{proof}
Recall that $q_*:\pi_1(M)\to \pi_1(G_0)$ is an epimorphism, it thus follows from Lemma \ref{lem:rg} that
\[ \rg(M,\phi)\geq \rg(\pi_1(G_0),r_*).\]
The lemma is now an immediate consequence of Lemma \ref{lem:rgfree}.
\end{proof}

Recall that $G_0$ is the underlying graph of a graph of spaces decomposition of $M$, with vertex spaces the components of $M\sm S\times (-1,1)$  and edge spaces those of $S$.  (We use the perspective on graphs of groups and spaces from \cite{ScoWa}; for definitions see p.~155 there.  See also \cite{Serre} and \cite{Tretkoff}).    This has an associated left action of $\pi$ on a tree $T$, without involutions, such that each vertex stabilizer is conjugate to $\pi_1(M)$ for some component $X$ of $\overline{M-S\times (-1,1)}$ and each edge stabilizer to $\pi_1(S_0)$ for some component $S_0$ of $S$ (see \cite[pp.~166--167]{ScoWa}.)

Using this we can now prove the non-trivial implication of Theorem \ref{mainthm} for closed hyperbolic $3$-manifolds.

\begin{theorem}\label{easy closed}
Let $(M,\phi)$ be a directed $3$-manifold where $M$ is a closed hyperbolic $3$-manifold.
If $\phi$ is non-fibered, then $\rg(M,\phi)>0$.
\end{theorem}

\begin{proof}
 Let $(M,\phi)$ be a directed $3$-manifold where $M$ is a closed hyperbolic $3$-manifold and where $\phi$ is non-fibered.
 We write $\pi=\pi_1(M)$ and we pick a surface $S$ of minimal complexity dual to $\phi$. Since $M$ is hyperbolic we can and will assume that no component of $S$ is a sphere or a torus. We denote by $G_0$ the graph which was defined above.

  On account of Lemma \ref{disconnected case} we may also assume that the graph $G_0$  has Euler characteristic $0$.  We will show below that $S$ is connected and non-separating; ie, $G_0 = G$.  Assuming this for the moment, let us prove the result.

Since $S$ is not a fiber surface, $\pi_1(S)$ is a quasi-Fuchsian subgroup of $\pi$ (see eg.~\cite{Bon86}).  Therefore by the main theorem of \cite{GMRS}, $\pi_1(S)$ has finite height in $\pi$ (cf.~the Corollary on \cite[p.~322]{GMRS}), so by Lemma \ref{cylinders vs height} the $\pi_1(M)$-action on the tree determined by $S$ is $k$-acylindrical for some $k\in \mathbb{N}$.  This action has quotient $G$, with one edge and vertex, so in particular it is minimal.  Since $M$ is hyperbolic and closed, $\pi$ is also non-cyclic, freely indecomposable, and finitely generated.

For each $n\in\mathbb{N}$, $\pi_n$ also acts on $T$, with quotient a graph $G_n$ with $n$ edges and vertices.  The action of $\pi_n$ inherits $k$-acylindricity from that of $\pi$, and since $\pi_n$  has finite index in $\pi$ its action is also minimal.  It now follows from Theorem \ref{thm:weidmann}  that
\begin{align}\label{gradient bound}  \rk(\pi_n) \geq \frac{n-1}{2k} + 1. \end{align}
We thus see that $\rg(M,\phi)>0$.

We return to showing that $G_0 = G$, assuming $\chi(G_0) =0$.
  Since $G_0$ has Euler characteristic zero, it is homotopy equivalent to its minimal-length closed edge path, call it $\gamma$.  Each edge of $G_0$ that is not in $\gamma$ is contained in a subtree $T_0$ of $G_0$ that intersects $\gamma$ at a single vertex $v_0$ with the property that $T_0-\{v_0\}$ is a component of $G_0-\{v_0\}$.  Since $T_0$ is a subtree, the component of $S$ corresponding to any edge in $T_0$ is nullhomologous.  Removing such a component thus reduces the complexity
 of $S$, so the fact that $S$ has minimal complexity  implies that there are none; ie, that $G_0=\gamma$.

We claim also that all edges of $G_0$ point in the same direction.  Note that identifying $\pi_1(G)$ with $\mathbb{Z}$ requires choosing an orientation for $e$.  This in turn gives an orientation to the interval fibers of each component of $S\times [-1,1]$ or, equivalently, an orientation to each edge of $G_0$.  If these do not all point in the same direction, at least one vertex $v_0$ of $\gamma$ is the initial vertex of each edge containing it.  The sum of the components of $S$ corresponding to these edges is trivial in homology, again contradicting the fact that $S$ has minimal complexity.  The claim follows, and implies that $G_0$ covers $G$.  But $p_*=\phi$ maps onto $\pi_1(G)$, so we must have $G_0=G$.
\end{proof}

\section{Proof of the main theorem for $3$-manifolds with empty or toroidal boundary}\label{section:proof2}

\subsection{Twisted Alexander polynomials}

In this section we quickly recall the definition of twisted Alexander polynomials.
This invariant was initially introduced by
 X. Lin \cite{Lin01}, M. Wada \cite{Wa94} and P. Kirk--C. Livingston \cite{KL99}. We refer to the survey paper \cite{FV10} for a detailed presentation.

Let $M$ be a $3$-manifold, let $\phi\in H^1(M;\Z)=\hom(\pi_1(M),\Z)$  and let $\a\co \pi_1(M)\to G$  be an epimorphism onto a finite group $G$.
We write $\pi=\pi_1(M)$.
 We can now define a left $\Q[\pi]$--module
structure on $\Q[G]\otimes_\Q \qt=:\Q[G]\tpm$ as follows:
\[  g\cdot (v\otimes p):=  (\a(g) \cdot v)\otimes (t^{\phi(g)}p), \]
where $g\in \pi$ and $ v\otimes p \in \Q[G]\otimes_\Q \qt = \Q[G]\tpm$.

Denote by $ \widetilde{M}$ the universal cover of $M$.
We then  use
the representation $\a\otimes \phi$ to regard $\Q[G]\tpm$ as a left $\Q[\pi]$--module.
The chain complex $C_*(\widetilde{M})$ is also a left $\Q[\pi]$--module via deck transformations.
Using the natural involution $g\mapsto g^{-1}$ on the group ring $\Q[\pi]$ we can  view $C_*(\widetilde{M})$ as a right $\Q[\pi]$--module.
 We can therefore consider
the  tensor products
\[ C_*^{\phi\otimes \a}(M;\Q[G]\tpm):=C_*(\tilde{M})\otimes_{\Q[\pi]}\Q[G]\tpm,\]
which form a complex of $\qt$-modules.
We then consider the $\qt$--modules
$$
H_*^{\phi\otimes \a}(M;\Q[G]\tpm) := H_*(C_*^{\phi\otimes \a}(M;\Q[G]\tpm)).
$$
When $\phi$ is understood, then we will drop it from the notation, similarly, if $\a$ is the trivial representation to $\gl(1,\Q)$,
then we will also drop it from the notation. We will later on also consider the  modules $H_*(M;\Q(t))$ and $H_*(M;\qt/(t^k-1))$
which are defined analogously.

 Since $M$ is compact and since $\qt$ is a PID we have an isomorphism
 \[ H_1^{\phi\otimes \a}(M;\Q[G]\tpm)\cong \bigoplus\limits_{i=1}^r\qt/p_i(t)\qt\]
 for some $p_i\in \qt$. We define the \emph{twisted Alexander polynomial} as follows
 \[ \Delta_{M,\phi}^\a:=\prod\limits_{i=1}^rp_i(t)\in \qt.\]
Note that $\Delta_{M,\phi}^\a\in \qt$  is well-defined up to multiplication by a unit in $\qt$.
We also adopt the convention that we drop $\a$ from the notation if $\a$ is the trivial representation to $\gl(1,\Q)$.

We will later on make use of the following two facts about (twisted) Alexander polynomials:

\begin{lemma}\label{lem1}
Let $(M,\phi)$ be a directed 3-manifold and let $\a\co \pi_1(M)\to G$ be an epimorphism onto a finite group.
We denote by $p\co \ti{M}\to M$ the corresponding finite cover. We write $\ti{\phi}:=p^*\phi$. Then
\[ \Delta_{M,\phi}^\a=\Delta_{\ti{M},\ti{\phi}}.\]
\end{lemma}

The lemma thus says that we can view a twisted Alexander polynomial of a directed 3-manifold $(M,\phi)$ as an untwisted Alexander polynomial of a corresponding cover of ${M}$. The lemma is a straightforward consequence of the Shapiro lemma. We refer to
\cite[Lemma~3.3]{FV08a} or \cite[Section~3]{FV10} for details.


\begin{lemma}\label{lem2}
Let $(M,\phi)$ be a directed 3-manifold with $\Delta_{M,\phi}=0$.
Let $n\in \Z$ and let $\pi_n:=\ker(\pi\xrightarrow{\phi}\Z\to \Z/n)$.
Then
\[ b_1(\pi_n)\geq n.\]
\end{lemma}

\begin{proof}
First note that the assumption that  $\Delta_{M,\phi}=0$ implies that  $H_1(M;\qt)\cong \qt\oplus H$
for some $\qt$-module $H$. It now follows from the Universal Coefficient Theorem that for any $n$ we have a short exact sequence
\[ \ba{rcl} 0&\to &H_1(M;\qt)\otimes_{\qt}\qt/(t^n-1)\\
&\to& H_1(M;\qt/(t^n-1))\\
&\to& \tor_{\qt}(H_0(M;\qt),\qt/(t^n-1))\to 0.\ea\]
Since $H_1(M;\qt)\cong \qt\oplus H$ it follows that
\[ \dim_\Q(H_1(M;\qt/(t^n-1)))\geq \dim_\Q(\qt/(t^n-1))=n.\]
Recall that we assumed that $\phi$ is primitive, which implies 
that the map $\pi\xrightarrow{\phi}\Z\to \Z/n$ is surjective.
We can thus apply Shapiro's lemma which in this case states that
\[ H_1(\pi_n;\Q)\cong H_1(\pi;\qt/(t^n-1))\cong H_1(M;\qt/(t^n-1)).\]
\end{proof}

\subsection{Twisted Alexander polynomials and fibered classes}

Let $(M,\phi)$ be a directed 3-manifold  and let   $\a\co \pi_1(M)\to G$ be an epimorphism onto a finite group.
 If $\phi$ is fibered then it was shown by many authors at varying levels of generality
 that $\Delta_{M,\phi}^\a$ is monic, in particular non-zero. We refer to \cite{Ch03,KM05,GKM05,Ki07,FK06,Fr13} for details.

In \cite{FV12}, extending earlier results in \cite{FV08b,FV11a,FV11b}, the following converse was proved.

\begin{theorem}\label{mainthmfib}
Let $(M,\phi)$ be a directed 3-manifold.
If $\phi \in H^1(M)$ is nonfibered, then there exists an epimorphism $\a\co \pi_1(M)\to G$ onto a finite group $G$ such that
\[ \Delta_{M,\phi}^\a =0.\]
\end{theorem}

The proof of this theorem relies heavily on the result of D. Wise \cite{Wi12a,Wi12b}
that subgroups of hyperbolic 3-manifolds which are carried by an embedded surface are separable. (See also \cite{AFW12} for precise references.)

\subsection{Proof of Theorem \ref{mainthm}}

As discussed in the introduction, the proof of Theorem \ref{mainthm} reduces to the proof of the
following.

\begin{theorem}\label{mainthm2}
If $(M,\phi)$ is a directed $3$-manifold which is not fibered, then  $\rg(M,\phi)>0$.
\end{theorem}

\begin{proof}
Let $(M,\phi)$ be a directed $3$-manifold such that $\phi$ is not fibered. We have to show that $\rg(M,\phi)>0$.
By Theorem \ref{mainthmfib}, there exists an epimorphism $\a\co \pi_1(M)\to G$ onto a finite group $G$ such that
\[ \Delta_{M,\phi}^\a =0.\]
We write $\pi:=\pi_1(M)$ and $\ti{\pi}:=\ker(\a)$ and we denote by $\ti{M}$ the cover corresponding to $\ti{\pi}$.
Note that $\phi(\ti{\pi})=d\Z$ for some $d\ne 0\in \Z$.
We write $\ti{\phi}:=\frac{1}{d}p^*(\phi)\in \hom(\ti{\pi},\Z)=H^1(\ti{M};\Z)$.
Note that $\ti{\phi}$ is a primitive class.

For any $n\in \N$ we furthermore write
\[ \pi_n:=\ker\{\pi\xrightarrow{\phi}\Z\to \Z/n\} \mbox{ and } \ti{\pi}_n:=\ker\{\ti{\pi}\xrightarrow{\ti{\phi}}\Z\to \Z/n\}.\]
We now have the following claim.

\begin{claim}
For any $n\in \N$ the group $\ti{\pi}_n$ is a subgroup of $\pi_{dn}$ of index at most $[\pi:\ti{\pi}]$.
\end{claim}

Note that for any $n\in \N$ we have
\[ \ti{\pi}_n:=\ker\{\ti{\pi}\xrightarrow{\ti{\phi}}\Z\to \Z/n\}=
\ker\{\ti{\pi}\xrightarrow{{\phi}}\Z\to \Z/dn\}.\]
We thus see that  the group $\ti{\pi}_n$ is indeed a subgroup of $\pi_{dn}$. 
We thus have the equalities
\[ [\pi:\ti{\pi}]\cdot [\ti{\pi}:\ti{\pi}_n]=[\pi:\ti{\pi}_n]=[\pi:\pi_{dn}]\cdot [\pi_{dn}:\ti{\pi}_n].\]
It now follows from $[\ti{\pi}:\ti{\pi}_n]=n$ and $[\pi:\pi_{dn}]=dn$ that 
\[ [\pi_{dn}:\ti{\pi}_n]=\frac{1}{d}[\pi:\ti{\pi}]\leq [\pi:\ti{\pi}].\]
This concludes the proof of the claim.

It follows from Lemmas \ref{lem1}  that $\Delta_{\ti{M},p^*\phi}=0$,
which in turn implies that $\Delta_{\ti{M},\ti{\phi}}=0$.
It now follows from Lemma \ref{lem2} that
\be \label{equ2}  b_1(\ti{\pi}_n)\geq n \mbox{ for any $n$.}\ee
For any $n$ we thus have by (\ref{equ:rs}) and (\ref{equ2}) that
\[ \frac{1}{n}\rk(\pi_n)\geq \frac{1}{dn}\rk(\pi_{dn})\geq\frac{1}{dn}\frac{1}{[\pi:\ti{\pi}]} \rk(\ti{\pi}_n)
\geq\frac{1}{dn}\frac{1}{[\pi:\ti{\pi}]} b_1(\ti{\pi}_n)
\geq\frac{1}{d}\frac{1}{[\pi:\ti{\pi}]}.\]
It thus follows that  $ \rg(M,\phi)> 0$ as desired.
\end{proof}

Recall that we assumed throughout the paper that $M$ is a compact $3$-manifold with empty or toroidal boundary.
The statement of Theorem \ref{mainthm} does not hold if $M$ has a  spherical boundary component.
Indeed, if $(M,\phi)$ is a fibered directed $3$-manifold, then deleting a 3-ball gives rise to a 3-manifold with the same fundamental group but which is no longer fibered. It is therefore reasonable to restrict ourselves to $3$-manifolds with no spherical boundary components.
Extending \textit{verbatim} the definition of rank gradient to this context, it is straightforward to see that the statement of Theorem \ref{mainthm} applies also to this slightly more general case:

\begin{lemma}\label{lem:extendmainthm}
Let $M$ be a compact $3$-manifold with no spherical boundary components and which has at least one non-toroidal boundary component.
Then $M$ is not fibered and for any primitive $\phi\in H^1(M;\Z)$ we have $\rg(M,\phi)>0$.
\end{lemma}

\begin{proof}

If $M$ fibers over $S^1$, then the boundary components also have to fiber over $S^1$, which means that all boundary components have to be tori.

Now let $M$ be a compact $3$-manifold  which has at least one non-toroidal boundary component $F$
and let  $\phi \in H^1(M;\Z)=\hom(\pi_1(M),\Z)$ be a primitive element. We have to show that $\rg(M,\phi)>0$.

We denote by $d\in \Z_{\geq 0}$
 the unique element such that $\phi(\pi_1(F))=d\Z$.
 We first suppose that $d>0$. Given $n\in \N$ we consider the finite cover $M_{nd}$ of $M$ corresponding to $\pi_1(M)\xrightarrow{\phi}\Z\to \Z/nd$
 and we furthermore consider the cover $F_{n}$ of $F$
corresponding to $\pi_1(F)\to \pi_1(M)\xrightarrow{\phi}d\Z\to d\Z/nd\cong \Z/n$. Note that by the assumption that $d$ is positive the cover $F_n$ is a connected cover of $F$.
By the multiplicativity of the Euler characteristic under finite covers we see that
\[ b_1(F_n)-2=n(b_1(F)-2).\]
Since $F$ is non-spherical and non-toroidal we see in particular that $b_1(F_n)\geq 2n$.

Note that $M_{nd}$ contains $d$ copies of $F_n$ as boundary components. By the standard half-live-half-die argument coming from Poincar\'e duality we deduce that
\[ b_1(M_{nd})\geq \frac{1}{2}(d\cdot b_1(F_n))=dn.\]
It is now obvious that $\rg(M,\phi)\geq 1$.

The case that $d=0$ is proved almost the same way. We leave the details to the reader.
\end{proof}

\begin{proof}[Proof of Theorem \ref{rankvsheegaard}]  For a compact $3$-manifold $M$ let $\widehat{M}$ be obtained from $M$ by filling all spherical boundary components with balls.  The inclusion map $M\to\widehat{M}$ takes Heegaard surfaces to Heegaard surfaces and induces an isomorphism of fundamental groups.  Moreover, every Heegaard surface for $\widehat{M}$ may be isotoped into $M$ by an innermost disk argument, giving a Heegaard surface there.  In particular, the Heegaard genus of $M$ equals that of $\widehat{M}$ as does the rank of $\pi_1$.

Since both spheres and balls lift to covers (having trivial $\pi_1$), if $M'\to M$ is a finite-degree cover then $\widehat{M'}$ is the cover of $\widehat{M}$ corresponding to $\pi_1 M' <\pi_1 M = \pi_1\widehat{M}$.  It follows that the rank and Heegaard gradients of any family of covers $\{M_n\to M\}$ may be computed in the corresponding family $\{\widehat{M}_n\to\widehat{M}\}$, reducing Theorem \ref{rankvsheegaard} to our prior results.\end{proof}

\section{Normal generating sets}

 In this section we will prove Theorem \ref{mainthmnormal}, whose statement we recall for the reader's convenience:\\

\noindent \textbf{Theorem 1.3}\emph{\bn
\item If  $(M,\phi)$  is a non-fibered directed $3$-manifold, then $\ker(\phi)$ admits a finite index subgroup with infinite normal rank.
\item There exists a non-fibered directed $3$-manifold $(M,\phi)$, such that $\ker(\phi)$ has finite normal rank.
\en}

\begin{proof}
We first note that if $\pi$ is any group, then any set of elements which normally generates $\pi$
is also a generating set of $H_1(\pi;\Z)$. It thus follows that
\[ n(\pi)\geq b_1(\pi).\]

If  $(M,\phi)$  is a non-fibered directed $3$-manifold, then
by Theorem \ref{mainthm}  there exists an epimorphism $\a\co \pi_1(M)\to G$ onto a finite group $G$ such that
\[ \Delta_{M,\phi}^\a =0.\]
We write $\ti{\pi}:=\ker(\a)$ and  we denote by $\ti{M}$ the cover of $M$ corresponding to $\ti{\pi}$.
As in the proof of Theorem \ref{mainthm2} we note that $\phi(\ti{\pi})=d\Z$ for some $d\ne 0\in \Z$
and we write $\ti{\phi}:=\frac{1}{d}p^*(\phi)\in \hom(\ti{\pi},\Z)=H^1(\ti{M};\Z)$.
Note that $\Delta_{M,\phi}^\a =0$ implies by Lemma \ref{lem1} that $\Delta_{\ti{M},\ti{\phi}}=0$.
This in turn is equivalent to saying that $H_1(\ti{M};\qt)$ is not $\qt$-torsion,
i.e.
\[ \dim(H_1(\mbox{$\ti{\phi}$-cover of $\ti{M}$};\Q))=\infty.\]
We thus see that  $b_1(\ker(\a\times \phi))=\infty$, i.e. $n(\ker(\a\times \phi))=\infty$.
Since $\ker(\a\times \phi)$ is a finite index subgroup of $\ker(\phi)$ this concludes the proof of (1).

We now turn to the proof of (2).
Let $(N,\psi)$ be a fibered directed $3$-manifold with $N\ne S^1\times D^2$.
We denote the fiber surface by $S$ and the monodromy by $\varphi$.
We can then identify $N$ with $(S\times [0,1])/(x,0)\sim (\varphi(x),1)$.
We pick an essential simple closed curve $C$ on $S\times \frac{1}{2}$ and we pick an open tubular neighborhood $\nu C$ of $C$
in $S\times (0,1)$. We furthermore pick a non-trivial knot $K\subset S^3$.
We then consider the $3$-manifold
\[ M:=(N\sm \nu C)\cup (S^3\sm \nu K)\]
where we glue the meridian of $K$ to a push-off of $C$ in $S\times \frac{1}{2}$ and where we glue the longitude of $K$ to a meridian of $C$.
We denote by $\phi\in H^1(M;\Z)$ the class which is dual to $S\times 0\subset M$.

We claim that $(M,\phi)$ has all the desired properties.
We denote by $\ti{M}$ the infinite cyclic cover  of $M$ corresponding to $\phi$.
Given $i\in \Z$ we write
\[ W_i:=\left((S\times [0,1]\sm \nu C)\cup (S^3\sm \nu K)\right)\times i.\]
Note that we can canonically identify $\ti{M}$ with
\[  \left(\bigcup\limits_{i\in \Z} W_i\right)/(x,i)\sim (\varphi(x),i+1).\]
Also note that $\ti{M}$ contains the incompressible tori $\partial \ol{\nu C}\times i$, in particular $\pi_1(\ti{M})$ is not a surface group.
It thus follows that $\phi$ is not a fibered class.

We now denote by $\Gamma$ the smallest normal subgroup of $\pi_1(\ti{M})=\ker(\phi)$ which contains $\pi_1(S\times 0)$.
We are done with the proof of Theorem \ref{mainthmnormal} once we showed that $\Gamma=\ker(\phi)$.
First note that $C$ and hence the meridian of $K\times 0$ lies in $\Gamma$. Since the meridian of $K$ normally generates
$\pi_1(S^3\sm \nu K)$ it follows that the longitude of $K\times 0$ also lies in $\Gamma$. It is now straightforward to see that
$\pi_1(W_0)\subset \Gamma$. This in particular implies that $\pi_1(S\times 1)$ lies in $\Gamma$.
But then the same argument as above shows that $\pi_1(W_1)\subset \Gamma$. Iterating this argument we see that $\pi_1(W_i)$ lies in $\Gamma$ for all $i\in \N$. Almost the same argument also shows that $\pi_1(W_i)$ lies in $\Gamma$ for all $i\in \Z_{\leq 0}$.
It now follows that $\pi_1(\ti{M})$ is contained in $\Gamma$.
\end{proof}


\end{document}